\documentclass[twocolumn]{autart}    %

\usepackage{graphicx}          %

\usepackage{amsmath}
\usepackage{amsfonts}       %
\usepackage{amssymb}
\usepackage{color}
\usepackage{enumerate}
\usepackage{booktabs}
\usepackage{makecell}
\usepackage{tikz}

\newcommand{\abs}[1]{\left\vert #1 \right\vert}
\newcommand{\ra}{\rightarrow}
\newcommand{\Real}{\mathbb{R}}

\newcommand{\BlackBox}{\rule{1.5ex}{1.5ex}}  %
\newenvironment{proof}{\par\noindent{\bf Proof\ }}{\hfill\BlackBox\\[2mm]}

\newif\ifarXiversion
\arXiversiontrue

\begin{document}

\begin{frontmatter}

\title{A Concise Lyapunov Analysis of Nesterov's Accelerated Gradient Method\thanksref{footnoteinfo}} %

\author{Jun Liu}\ead{j.liu@uwaterloo.ca}

\address{Department of Applied Mathematics\\
       University of Waterloo\\
       Waterloo, Ontario N2L 3G1, Canada\\}  %

\begin{keyword}                           %
Optimization algorithms; Nesterov's accelerated gradient method; Lyapunov analysis%
\end{keyword}                             %

\begin{abstract}                          %
Convergence analysis of Nesterov's accelerated gradient method has attracted significant attention over the past decades. While extensive work has explored its theoretical properties and elucidated the intuition behind its acceleration, a simple and direct proof of its convergence rates is still lacking. We provide a concise Lyapunov analysis of the convergence rates of Nesterov's accelerated gradient method for both general convex and strongly convex functions.
\end{abstract}

\end{frontmatter}

\section{Introduction}

The resurgence of machine learning over the past decades has renewed interest in first-order optimization methods with provable fast convergence rates \cite{sutskever2013importance,polyak1964some,nesterov1983method}. Among them, Nesterov’s accelerated gradient method \cite{nesterov1983method,nesterov2013introductory} has gained significant attention due to its provable acceleration on general convex functions beyond quadratics. A special focus has been on using dynamical system tools \cite{su2016differential,polyak2017lyapunov,dobson2024connections,wilson2021lyapunov} and control-theoretical methods \cite{lessard2016analysis,padmanabhan2024analysis} for the analysis and design of such algorithms.

In the standard textbook \cite{nesterov2013introductory} by Nesterov, the convergence analysis of accelerated gradient methods is conducted using a technique known as estimating sequences. These are essentially auxiliary comparison functions used to prove the convergence rates of optimization algorithms. As pointed out in \cite{wilson2021lyapunov}, estimating sequences are usually constructed inductively and can be difficult to understand and apply. This motivated the Lyapunov analysis in \cite{wilson2021lyapunov}, which aims to unify the analysis of a broad class of accelerated algorithms. Despite this comprehensive work, to the best knowledge of the author, a simple and direct Lyapunov analysis of the original scheme of Nesterov’s accelerated gradient method is still lacking.

In this work, we provide a streamlined and concise Lyapunov analysis of Nesterov’s accelerated gradient method for both general convex and strongly convex functions. We believe this analysis offers a clear and unified perspective on its convergence behavior, is more accessible, and can hopefully provide insights for analyzing Nesterov’s accelerated gradient method in more general settings, such as proving  acceleration in stochastic optimization beyond quadratics \cite{assran2020convergence}.

The use of Lyapunov functions (sometimes called potential functions) for the analysis of optimization algorithms has been popularized over the years by various authors. In addition to \cite{wilson2021lyapunov} mentioned above, readers are referred to \cite[Section~5]{bansal2019potential}, \cite[Section~4]{d2021acceleration}, \cite[Section 4]{chambolle2016introduction}, and references therein. Readers are also referred to dissipative theory, which is closely related to Lyapunov analysis, in the context of Nesterov’s accelerated method \cite{hu2017dissipativity}. We will add several remarks throughout the paper to point readers to closely related work.

\section{Preliminaries}

We recall some basic definitions and results for unconstrained smooth optimization, which can be found in \cite[Chapter 2]{nesterov2013introductory}. We use \( \abs{x} \) to denote the Euclidean norm of $x\in\Real^n$ and \( x \cdot y\) to denote the inner product of $x,y\in\Real^n$. The gradient of a differentiable function $f:\,\Real^n\ra\Real$ is denoted by $\nabla f$. 

\begin{defn}[$L$-smoothness]\label{def:l-smooth}\em
    A continuously differentiable function $f:\,\Real^n\ra\Real$ is said to be \textit{$L$-smooth}, if its gradient is Lipschitz continuous with Lipschitz constant $L>0$, i.e., 
    $$
    \abs{\nabla f(x)- \nabla f(y)} \le L\abs{x-y},
    $$
    for all $x,y\in\Real^n$.
\end{defn}

\begin{defn}[Convexity]\label{def:convexity}\em
    A differentiable function $f:\,\Real^n\ra\Real$ is said to be \textit{convex} on $\Real^n$, if it satisfies
    \begin{equation}
        \label{eq:convexity}
        f(y) \ge f(x) + \nabla f(x) \cdot (y-x),
    \end{equation}
    for all $x,y\in\Real^n$.
\end{defn}

\begin{defn}[Strong convexity]\label{def:strong-convexity}\em
    A differentiable function $f:\,\Real^n\ra\Real$ is said to be \textit{$\mu$-strongly convex} on $\Real^n$ for some $\mu> 0$, if it satisfies
    \begin{equation}
        \label{eq:strong_convexity}
        f(y) \ge f(x) + \nabla f(x) \cdot (y-x) + \frac{\mu}{2}\abs{y-x}^2,
    \end{equation}
    for all $x,y\in\Real^n$. 
\end{defn}

The following well-known inequality for $L$-smooth functions (also known as the descent lemma) will be used later. 

\begin{lem}\cite[Lemma 1.2.3]{nesterov2013introductory}\label{lem:l-smooth}
    If $f:\,\Real^n\ra\Real$ is $L$-smooth, then 
    $$
    f(y) \le f(x) + (y-x)\cdot \nabla f(x) + \frac{L}{2}\abs{x-y}^2
    $$
    for all $x,y\in\Real^n$.
\end{lem}

An immediate consequence of Lemma \ref{lem:l-smooth} is that, with a gradient descent step 
$
y = x - \alpha \nabla f(x)   
$ 
on an $L$-smooth function $f$, we have
\begin{align}
f(y) & \le f(x) + (-\alpha + \frac{L\alpha^2}{2}) \abs{\nabla f(x)}^2, \label{eq:gd-step}
\end{align}
which, in particular, leads to
\begin{align}
f(y) & \le f(x) - \frac{1}{2L}\abs{\nabla f(x)}^2,    \label{eq:gd-decrease}
\end{align}
when $\alpha=\frac{1}{L}$.

\section{General convex functions}

Consider Nesterov's accelerated gradient (NAG) method: 
\begin{equation}
    \label{eq:nesterov}
    \begin{aligned}
    x_{k+1} & = y_k - \alpha_k \nabla f(y_k),\\
    y_k & = x_k + \beta_k (x_k - x_{k-1}),\quad k\ge 0,
    \end{aligned}
\end{equation}
where $x_0=x_{-1}$ are given as initial conditions and $\alpha_k$ and $\beta_k$ are parameters to be chosen. 

\begin{thm}\label{thm:weakly_convex}
Let $f:\Real^n\ra\Real$ be $L$-smooth and convex. Let $x^*$ be a minimum of $f$ and $f^*=f(x^*)$. Then the NAG algorithm (\ref{eq:nesterov}) with $\alpha_k=\alpha \in (0, \frac{1}{L}]$ and $\beta_k=\frac{k-1}{k+r-1}$, $r\ge 3$, leads to    
\begin{equation}
    \label{eq:bound_weakly_convex}
    f(x_k) - f^* \le \frac{(r-1)^2\abs{x_0-x^*}^2}{2\alpha(k+r-2)^2},\quad k\ge 0.
\end{equation}
\end{thm}

\begin{proof}
Define $a_0=0$ and $a_k=\frac{k+r-2}{r-1}$, $k=1,2,\ldots$. We can rewrite $\beta_k = \frac{\frac{k+r-2}{r-1}-1}{\frac{k+r-1}{r-1}} = \frac{a_k-1}{a_{k+1}}$, $k\ge 1$. Introduce 
\begin{equation}\label{eq:pk}
p_k: = (a_k-1)(x_k-x_{k-1}) \stackrel{\eqref{eq:nesterov}}{=} a_{k+1}(y_k-x_k),\quad k\ge 1.    
\end{equation}
Let $p_0=0$. Then equation (\ref{eq:pk}) also holds for $k=0$ by the initial condition $x_0=x_{-1}$ and $y_0=x_0$. By (\ref{eq:pk}) and (\ref{eq:nesterov}), it is straightforward to verify that
\begin{align*}
p_{k+1} + x_{k+1} 
& = p_k + x_k - a_{k+1}\alpha \nabla f(y_k),\quad k\ge 0.
\end{align*}
It follows that
\begin{align}
&\abs{ p_{k+1} + x_{k+1} - x^*}^2 \notag\\
&= \abs{ p_{k} + x_{k} - x^*}  - 2\alpha a_{k+1}( p_{k} + x_{k} - x^*)\cdot \nabla f(y_k)\notag\\
&\, + a_{k+1}^2\alpha^2 \abs{\nabla f(y_k)}^2\notag\\
& = \abs{ p_{k} + x_{k} - x^*} - 2\alpha  a_{k+1}(a_{k+1}-1)(y_k-x_k)\cdot \nabla f(y_k) \notag\\
&\, -  2\alpha  a_{k+1}(y_k - x^*)\cdot \nabla f(y_k) + a_{k+1}^2\alpha^2 \abs{\nabla f(y_k)}^2.\label{eq:est1}
\end{align}
By (\ref{eq:gd-step}), we have 
\begin{equation}
\label{eq:est_new}
f(x_{k+1}) \le f(y_k) + (-\alpha + \frac{L\alpha^2}{2})\abs{\nabla f(y_k)}^2. 
\end{equation}
By convexity of $f$ (Definition \ref{def:convexity}),  we have
\begin{equation}
\label{eq:est3}
(y_k-x_k)\cdot \nabla f(y_k) \ge f(y_k) - f(x_k). 
\end{equation}
and
\begin{equation}
\label{eq:est4}
(y_k - x^*)\cdot \nabla f(y_k) \ge f(y_k) - f^*.
\end{equation}
Substituting (\ref{eq:est3}) and (\ref{eq:est4}) into (\ref{eq:est1}) gives
\begin{align}
    &\abs{ p_{k+1} + x_{k+1} - x^*}^2 \notag\\
&\le \abs{ p_{k} + x_{k} - x^*}^2 - 2\alpha a_{k+1}(a_{k+1}-1) (f(y_k) - f(x_k))\notag\\
&\quad -  2\alpha a_{k+1} (f(y_k) - f^*)  + a_{k+1}^2\alpha^2 \abs{\nabla f(y_k)}^2\notag\\
&\le \abs{ p_{k} + x_{k} - x^*}^2 - 2\alpha a_{k+1}^2 (f(y_k) - f^*) \notag\\
&\quad + 2\alpha a_k^2 (f(x_k) - f^*)  + a_{k+1}^2\alpha^2 \abs{\nabla f(y_k)}^2\notag\\
&\le \abs{ p_{k} + x_{k} - x^*}^2 - 2\alpha a_{k+1}^2 (f(x_{k+1}) - f^*) \notag\\
&\quad + 2\alpha a_k^2 (f(x_k) - f^*) \notag\\
&\quad + a_{k+1}^2\alpha^2(- 1 + L\alpha)\abs{\nabla f(y_k)}^2 \notag  \\
&\le \abs{ p_{k} + x_{k} - x^*}^2 - 2\alpha a_{k+1}^2 (f(x_{k+1}) - f^*) \notag\\
&\quad + 2\alpha a_k^2 (f(x_k) - f^*),
\label{eq:est5}
\end{align}
where the first inequality follows from substituting (\ref{eq:est3}) and (\ref{eq:est4}) into (\ref{eq:est1}), the second inequality used $a_{k+1}(a_{k+1}-1)\le a_k^2$ for all $k\ge 0$ (which can be verified directly), the third inequality follows from (\ref{eq:est_new}), and the fourth inequality holds because $\alpha\le \frac{1}{L}$.

Inequality (\ref{eq:est5}), by rearrangement, shows that  
$$
V_{k} = \abs{ p_{k} + x_{k} - x^*}^2 + 2\alpha a_{k}^2 (f(x_{k}) - f^*)
$$
is a \textit{Lyapunov function} in the sense that it is nonincreasing with respect to $k$ along iterates of (\ref{eq:nesterov}). It follows that
\begin{align*}
2\alpha a_{k}^2 (f(x_{k}) - f^*)\le V_k \le V_0 = \abs{x_0-x^*}^2, \\
&
\end{align*}
which gives (\ref{eq:bound_weakly_convex}). 
\end{proof}

\begin{rem}
The proof of Theorem \ref{thm:weakly_convex} essentially follows the argument of the original proof in \cite{nesterov1983method}. We made the choice of step size more explicit, following \cite{su2016differential}. Additionally, we reformulated the more general step size choice, \(\beta_k\), considered in \cite{su2016differential}, to align with the format of the original proof in \cite{nesterov1983method}. We believe this streamlined proof will be of interest to readers.
\end{rem}

\begin{cor}
    The special case of Theorem \ref{thm:weakly_convex} with $\alpha=\frac{1}{L}$ and $\beta_k=\frac{k-1}{k+2}$ gives 
    \begin{equation}
    \label{eq:bound_weakly_convex_1}
    f(x_k) - f^* \le \frac{2L\abs{x_0-x^*}^2}{(k+1)^2},\quad k\ge 0.
\end{equation}
\end{cor}

Here, acceleration refers to the \(1/k^2\) convergence rate, compared to the \(1/k\) convergence rate achieved by standard gradient descent \cite{nesterov2013introductory}.

\section{Strongly convex functions}

\begin{thm}
    \label{thm:strong_convex}
    Let $f:\Real^n\ra\Real$ be $L$-smooth and $\mu$-strongly convex. Let $x^*$ be a minimum of $f$ and $f^*=f(x^*)$. Then the NAG algorithm (\ref{eq:nesterov}) with $\alpha_k=\frac{1}{L}$ and $\beta_k=\frac{\sqrt{\kappa}-1}{\sqrt{\kappa}+1}$, where $\kappa=\frac{L}{\mu}$, leads to 
    \begin{equation*}
        f(x_{k})-f^* \le \left(1-\frac{1}{\sqrt{\kappa}}\right)^k \left(f(x_0)-f^* + \frac{\mu}{2}\abs{x_0-x^*}^2\right)
    \end{equation*}
    for all $k\ge 0$.
\end{thm}

\begin{proof}
Introduce 
\begin{equation}
    \label{eq:vk}
v_k = (\sqrt{\kappa}+1)y_k - \sqrt{\kappa} x_k,\quad k\ge 0.
\end{equation}
By the initial condition $x_0=x_{-1}$ and (\ref{eq:nesterov}), we have $v_0=x_0=y_0$. 
By (\ref{eq:vk}) and (\ref{eq:nesterov}), it is also straightforward to verify that 
\begin{equation}
    \label{eq:vk+1}
v_{k+1} =v_k + \frac{1}{\sqrt{\kappa}}(y_k - v_k) - \frac{1}{\mu\sqrt{\kappa}} \nabla f(y_k).
\end{equation}
Define 
$$
V_k=f(x_{k})-f^* + \frac{\mu}{2}\abs{v_k-x^*}^2. 
$$ 
It follows that
\begin{align}
    V_{k+1} & = f(x_{k+1})-f^* + \frac{\mu}{2}\abs{v_{k+1}-x^*}^2 \notag\\
    &\stackrel{\eqref{eq:gd-decrease}}{\leq} f(y_k) - \frac{1}{2L}\abs{\nabla f(y_k)}^2 - f^* + \frac{\mu}{2}\abs{v_{k+1}-x^*}^2 \notag\\
    & \stackrel{\eqref{eq:vk+1}}{=} f(y_k) - \frac{1}{2L}\abs{\nabla f(y_k)}^2 - f^* \notag \\
    &\quad + \frac{\mu}{2}\abs{v_k-x^*}^2 + \frac{\mu}{2\kappa}\abs{y_k-v_k}^2 + \frac{1}{2\mu\kappa}\abs{\nabla f(y_k)}^2 \notag \\
    &\quad  - \frac{1}{\kappa}(y_k-v_k)\cdot \nabla f(y_k) + \frac{\mu}{\sqrt{\kappa}}(v_k-x^*)\cdot (y_k-v_k)\notag\\
    &\quad - \frac{1}{\sqrt{\kappa}}(v_k-x^*)\cdot \nabla f(y_k)
    \label{eq:v_est1}
\end{align}
Rewrite\footnote{This elementary equality  
\[
(a - b) \cdot (c - a) = \frac{1}{2} \left( \abs{c - b}^2 - \abs{c - a}^2 - \abs{a - b}^2 \right)
\]  
is sometimes called the three-point identity. To see an easy proof, write \( c - b = (a - b) + (c - a) \) and compute \( \abs{c - b}^2 \).}
\begin{align}
 & (v_k-x^*)\cdot (y_k-v_k) \notag\\
 & \quad = \frac12(\abs{y_k-x^*}^2 - \abs{v_k-x^*}^2  - \abs{y_k-v_k}^2),\label{eq:v_est2}
\end{align}
and
\begin{align}
 & -(v_k-x^*)\cdot \nabla f(y_k) \notag\\
 & \quad = (y_k - v_k)\cdot \nabla f(y_k)  + (x^*-y_k)\cdot \nabla f(y_k). \label{eq:v_est3}
\end{align}
By $\mu$-strong convexity of $f$ (Definition \ref{def:strong-convexity}), we have 
\begin{align}
 &  (x^* - y_k)\cdot \nabla f(y_k) \le f^* - f(y_k) - \frac{\mu}{2}\abs{y_k-x^*}^2.  \label{eq:v_est4}
\end{align}
Substituting (\ref{eq:v_est4}) into (\ref{eq:v_est3}), then substituting (\ref{eq:v_est3}) and (\ref{eq:v_est2}) into (\ref{eq:v_est1}), and canceling out terms, we obtain
\begin{align}
    V_{k+1} & \le (1-\frac{1}{\sqrt{\kappa}}) (f(y_k) - f^*) + (1-\frac{1}{\sqrt{\kappa}})  \frac{\mu}{2}\abs{v_k-x^*}^2 \notag\\
    &\quad + (\frac{\mu}{2\kappa} - \frac{\mu}{2\sqrt{\kappa}} ) \abs{y_k-v_k}^2 \notag\\
    &\quad - (\frac{1}{\kappa}-\frac{1}{\sqrt{\kappa}})(y_k - v_k)\cdot \nabla f(y_k).
    \label{eq:v_est5}
\end{align}
Applying convexity of $f$ (Definition \ref{def:convexity}), we have
\begin{align}
f(y_k) \le f(x_k) - (x_k-y_k)\cdot \nabla f(y_k).
\label{eq:v_est6}
\end{align}
Putting (\ref{eq:v_est6}) into (\ref{eq:v_est5}) and noticing that $\frac{\mu}{2\kappa} - \frac{\mu}{2\sqrt{\kappa}}\le 0$ because $\kappa\ge 1$, we have
\begin{align}
    V_{k+1} & \le (1-\frac{1}{\sqrt{\kappa}}) (f(x_k) - f^*) + (1-\frac{1}{\sqrt{\kappa}})  \frac{\mu}{2}\abs{v_k-x^*}^2 \notag\\
    - \bigg[(\frac{1}{\kappa}  & -\frac{1}{\sqrt{\kappa}})(y_k - v_k)+(1-\frac{1}{\sqrt{\kappa}})(x_k-y_k)\bigg]\cdot \nabla f(y_k),\notag
\end{align}
where the last term vanishes because by (\ref{eq:vk}) we have $y_k - v_k=\sqrt{\kappa}(x_k-y_k)$. We have proved that $V_{k+1}\le (1-\frac{1}{\sqrt{\kappa}}) V_k$ for $k\ge 0$. The conclusion follows immediately. 
\end{proof}

For strongly convex functions, acceleration of convergence by Nesterov's method (\ref{eq:nesterov}) refers to improving the convergence rates of the standard gradient descent method, \((1 - 1/\kappa)^k\) (with step size \(\alpha = 1/L\)) or \(((\kappa - 1)/(\kappa + 1))^{2k} = (1 - 2/(\kappa + 1))^{2k}\) (with step size \(\alpha = 2/(L+\mu)\)), to \((1 - 1/\sqrt{\kappa})^k\), as shown in Theorem \ref{thm:strong_convex} \cite{nesterov2013introductory}, which is a significant improvement for large \(\kappa\).

\begin{rem}
The proof is inspired by the estimating sequence argument in \cite{nesterov2013introductory}, which treats both general (weakly) and strongly convex functions in a unified fashion. While estimating sequences are powerful, their construction is inductive, and the analysis of the algorithm becomes entangled with its design. Lyapunov analysis is well known for effectively analyzing the stability of dynamical systems \cite{khalil2002nonlinear}. Since optimization algorithms can be viewed as discrete-time dynamical systems, it is not surprising that Lyapunov functions can be used to analyze their convergence \cite{polyak2017lyapunov,wilson2021lyapunov}. In particular, \cite{wilson2021lyapunov} provides a comprehensive study of Lyapunov analysis for accelerated methods in optimization as discretizations of continuous-time dynamical systems. Despite this comprehensive work, we believe a concise and direct Lyapunov analysis can be beneficial in revealing the convergence behavior of algorithm (\ref{eq:nesterov}). We encourage interested readers to contrast the analysis in \cite{wilson2021lyapunov}, particularly Proposition 13, as well as the proof using estimating sequences in \cite[Section 2.2]{nesterov2013introductory} with the proof of Theorem \ref{thm:strong_convex} presented here.
\end{rem}

\begin{rem}
Dissipativity theory for Nesterov's accelerated methods has been investigated in \cite{hu2017dissipativity}. A key benefit of this framework is that it enables the construction of Lyapunov functions, either numerically or analytically, by solving small semidefinite programs. In \cite{hu2017dissipativity}, the authors provided analytical solutions to LMIs for both the general and strongly convex cases of Nesterov’s method; see \cite[Theorem~9]{hu2017dissipativity} and \cite[Theorem~4]{hu2017dissipativity}, respectively. While their results are more general and parametrized by LMI solutions, Theorems~\ref{thm:weakly_convex} and~\ref{thm:strong_convex} in this paper provide concise and self-contained proofs for the basic forms of Nesterov’s scheme. These may be of interest to readers seeking simplified derivations or alternative insights.
\end{rem}

\section{Conclusions}

We have presented a concise Lyapunov analysis of the celebrated Nesterov's accelerated gradient method. We hope this simplified analysis is not only of tutorial value but also provides insights into the analysis of Nesterov's accelerated gradient method in other settings, such as proving their acceleration in the stochastic setting beyond quadratic functions \cite{assran2020convergence}, both in expectation \cite{bottou2018optimization} and in the almost sure sense \cite{liu2024almost}.

\begin{ack}                               %
This work was partially supported by the NSERC of Canada and the CRC program. The author is grateful to Nicolas Boumal for identifying a typo in an earlier version that caused an issue in the proof of Theorem \ref{thm:weakly_convex}. The author is also grateful to reviewers of an earlier version of this paper, who provided additional references. 
\end{ack}

\bibliographystyle{plain}        %
\bibliography{refs} 

\end{document}